\documentclass[preprint,times]{elsarticle}

\usepackage[utf8]{inputenc}

\usepackage[english]{babel}
\usepackage{graphicx}
\usepackage{hyperref}
\usepackage{amsmath,amsthm,amssymb}
\usepackage{enumerate}
\usepackage{cleveref} 
\usepackage{tikz}

\def\A{\mathcal{A}}

\def\N{\mathbb{N}}

\newtheorem{theorem}[]{Theorem}
\newtheorem{corollary}[theorem]{Corollary}
\newtheorem{lemma}[theorem]{Lemma}

\newtheorem{definition}[theorem]{Definition}

\theoremstyle{remark}

\newtheorem{example}[theorem]{Example}

\crefname{theorem}{Theorem}{Theorems}
\crefname{corollary}{Corollary}{Corollaries}
\crefname{example}{Example}{Examples}
\crefname{lemma}{Lemma}{Lemmas}
\crefname{proposition}{Proposition}{Propositions}
\crefname{definition}{Definition}{Definitions}
\crefname{example}{example}{examples}

\begin{document}

\begin{frontmatter}


\title{Synchronizing delay for binary uniform morphisms}
\author[kk]{Karel Klouda}
\ead[kk]{karel.klouda@fit.cvut.cz}
\author[km]{Kate\v{r}ina Medkov\'{a}}
\address[kk]{
Faculty of Information Technology, Czech Technical University in
Prague, Th\'{a}kurova 9, 160 00, Prague 6}
\address[km]{
Faculty of Nuclear Sciences and Physical Engineering, Czech Technical University in
Prague,B\v{r}ehov\'{a} 7, 115 19, Prague 1}



\begin{abstract}
Circular D0L-systems are those with finite synchronizing delay. We introduce a tool called graph of overhangs which can be used to find the minimal value of synchronizing delay of a given D0L-system. By studying the graphs of overhangs, a general upper bound on the minimal value of a synchronizing delay of a circular D0L-system with a binary uniform morphism is given.
\end{abstract}

\end{frontmatter}

\section{Introduction}

Circular codes are a classical notion studied in theory of codes~\cite{BePeRe}. A set $X$ of finite words is a code if each word in $X^+$ has a unique decomposition into words from $X$. If we slightly modify the requirement of uniqueness, we get the definition of a circular code: $X$ is a circular code if each word in $X^+$ written in a circle has a unique decomposition into words from $X$.

An analogue to codes in the family of D0L-systems are D0L-systems that are injective on the set of all factors of their languages. Circularity is defined as slightly relaxed injectivity: a D0L-system is circular if long enough factors of its language have a unique preimage (under the respective morphism) in the language except for some prefix and suffix bounded in length by some constant. This constant is called a synchronizing delay and it is studied in this paper.

In the case of D0L-systems circularity is connected with repetitiveness. As stated in~\cite{MiSe}, a non-circular D0L-system is repetitive, i.e., for each $k \in \N$ there exists a word $v$ such that $v^k$ is a factor of the language. In fact, if a D0L-system is not pushy (which is always true if the morphism is uniform and the language is infinite), then circularity is equivalent to non-repetitiveness~\cite{KlSt13a}.

As explained by Cassaigne in~\cite{Ca94}, knowledge of the value of the synchronizing delay can be very helpful when analysing the structure of bispecial factors in languages of D0L-systems. This idea was further developed by one of the authors in~\cite{Kl12}, where an algorithm for generating all bispecial factors is given. This algorithm works for circular and non-pushy D0L-systems and its computational complexity depends on the value of the synchronizing delay. This fact and the absence of any known bound on the value of synchronizing delay is the main motivation of the present work.

Unfortunately, it seems it is not easy to find such a bound. Therefore we focus on the simplest case: a circular D0L-system with binary $k$-uniform morphism with $k \geq 2$. Using the notion of the graph of overhangs introduced in Subsection~\ref{sec:graph_of_overhangs}, we prove the following result. The details of the proof are given in Section~\ref{sec:the-proof}.
\begin{theorem}\label{thm:main_result}
	If the morphism $\varphi$ of a circular D0L-system $(\{a,b\}, \varphi, a)$ is $k$-uniform, then the minimum value of its synchronizing delay, denoted by $Z_\text{min}$, is bounded as follows:
	\begin{enumerate}[(i)]
		\item $Z_{\text{min}} \leq 8$ if $k = 2$,
		\item $Z_{\text{min}} \leq k^2 + 3k - 4$ if $k$ is an odd prime number,
		\item $Z_{\text{min}} \leq k^2\left(\frac{k}{d} - 1\right) + 5k - 4$ otherwise,
	\end{enumerate}
	where number $d$ is the least divisor of $k$ greater than $1$.
\end{theorem}

\section{Preliminaries}

A finite set of symbols is an \emph{alphabet}, denoted by $\A$. The set of all finite words over $\A$ is denoted by $\A^*$, the \emph{empty word} is $\varepsilon$ and $\A^+ = \A^* \setminus \{\varepsilon\}$.
If a word $u \in \A^*$ is a concatenation of three words $x, y$ and $z$ from $\A^*$, i.e., $u = xyz$, the word $x$ is a \emph{prefix} of $u$, $y$ its \emph{factor} and $z$ a \emph{suffix}. We put $x^{-1}u = yz$ and $uz^{-1} = xy$. The \emph{length} of the word $u$ equals the number of letters in $u$ and is denoted by $|u|$; $|u|_a$ denotes the number of occurrences of a letter $a$ in $u$.

A mapping $\varphi: \A^* \to \A^*$ is a \emph{morphism} if for every $v,u \in \A^*$ we have $\varphi(vu) = \varphi(v)\varphi(u)$. A triplet $G = (\A, \varphi, w)$ is a \emph{D0L-system}, if $\varphi$ is a morphism on $\A$ and $w \in \A^+$. The word $w$ is called an \emph{axiom}. 
The \emph{language of $G$} is the set $L(G) = \{ \varphi^n(w) \colon n \in \N \}$. The set of all factors of elements of $L(G)$ is denoted by $S(L(G))$.
The alphabet is always considered to be the minimal alphabet necessary, i.e., $\A \cap S(L(G)) = \A$.

A D0L-system $G = (\A, \varphi, w)$ is \emph{injective on $S(L(G))$} if for every $u,v \in S(L(G))$, $\varphi(u) = \varphi(v)$ implies that $u = v$.
It is clear that if $\varphi$ is injective, then $G$ is injective. If $\varphi$ is \emph{non-erasing}, i.e., $\varphi(a) \neq \varepsilon$ for all $a \in \A$, then $G$ is a \emph{propagating D0L-system}, shortly \emph{PD0L-system}.

Given a D0L-system $G = (\A, \varphi, w)$, we say that a letter $a$ is \emph{bounded} if the set $\{ \varphi^n(a) \colon n \in \N \}$ is finite.
If a letter is not bounded, it is \emph{unbounded}. The system $G$ is \emph{pushy} if $S(L(G))$ contains infinitely many factors containing bounded letters only.

A D0L-system $G$ is \emph{repetitive} if for any $k \in \N$ there is a non-empty word $v$ such that $v^k$ is a factor from $S(L(G))$.
By~\cite{EhRo83}, any repetitive D0L-system is \emph{strongly repetitive}, i.e., there is a non-empty word $v$ such that $v^k$ is a factor for all $k \in \N$. We say that $G$ is \emph{unboundedly repetitive} if there is $v$ containing at least one unbounded letter such that $v^k$ is a factor for all $k \in \N$.

\subsection{Circular D0L-systems}


In~\cite{Ca94}, a circular D0L-system is defined using the notion of synchronizing point (see Section 3.2 in~\cite{Ca94} for details). We give here an equivalent definition employing the notion of interpretation.
\begin{definition}
Let $G = (\A,\varphi, w)$ be a PD0L-system and $u \in S(L(G))$.
A triplet $(p,v,s)$, where $p,s \in \A^*$ and $v \in S(L(G))$, is an \emph{interpretation of the word $u$} if $\varphi(v) = pus$.
\end{definition}

\begin{definition}
Let $G = (\A,\varphi, w)$ be a PD0L-system. We say that two interpretations $(p,v,s)$ and $(p',v',s')$ of a word $u \in S(L(G))$ are \emph{synchronized at position $k$} if there exist indices $i$ and $j$ such that
$$
	\varphi(v_1\cdots v_i) = p u_1 \cdots u_k \quad \text{ and } \quad \varphi(v'_1\cdots v'_j) = p' u_1 \cdots u_k
$$
with $v = v_1\cdots v_n \in \A^n$, $v' = v'_1 \cdots v'_m \in \A^m$ and $u = u_1 \cdots u_\ell \in \A^\ell$ (if $k = 0$, we put $u_1 \cdots u_k = \varepsilon$) Two interpretations that are not synchronized at any position are called \emph{non-synchronized}.

We say that a word $u \in S(L(G))$ has a \emph{synchronizing point} at position $k$ with $0 \leq k \leq |u|$ if all its interpretations are pairwise synchronized at position $k$.
\end{definition}

\begin{definition}
Let $G = (\A,\varphi, w)$ be a PD0L-system injective on $S(L(G))$. We say that $G$ is circular if there is a positive integer $Z$, called a \emph{synchronizing delay}, such that any $u$ from $S(L(G))$ longer than $Z$ has a synchronizing point. The minimal constant $Z$ with this property is denoted by $Z_{\text{min}}$.
\end{definition}

By the results from~\cite{MiSe, KlSt13a}, non-circular systems are repetitive (and by \cite{EhRo83} also strongly repetitive). In fact, a D0L-system injective on $S(L(G))$ is not circular if and only if it is \emph{unboundedly repetitive}~\cite{KlSt13a}, i.e., there exists $v$ containing an unbounded letter such that $v^k \in S(L(G))$ for all $k \in \N$. Since this property can be checked by a simple algorithm~\cite{KlSt13}, we can easily verify whether a given D0L-system  injective on $S(L(G))$ is circular or not.

The notion of circularity is inspired by the notion of circular code:
\begin{definition}
A subset $X$ of $\A^*$ is called a \emph{code} over alphabet $\A$ if for any word $v \in X^+$ there are uniquely given a number $n$ and words $x_1, x_2, \ldots, x_n$ from $X$ so that $v = x_1x_2\ldots x_n$.

The set $X$ is a \emph{circular code} over $\A$ if for all $n, m \geq 1$, $x_1, \ldots, x_n, y_1, \ldots , y_m \in X$, $p \in \A^*$ and $s \in \A^+$ it holds that:
$$
(sx_2x_3 \cdots x_np = y_1y_2 \cdots y_m \ \text{and} \ x_1 = ps)\Longrightarrow (n=m \, , \  p = \epsilon \ \text{and} \ x_i = y_i \ \forall \ i= 1, \ldots , n)\, .
$$
\end{definition}

\subsection{Graph of overhangs} \label{sec:graph_of_overhangs}

Here we introduce the basic tool we use to prove the main result of this paper: graphs of overhangs. To understand the motivation of the definition let us consider a D0L-system $G = (\A, \varphi, w)$ and a word $u$ from $S(L(G))$ that has two non-synchronized interpretations $(p_1,x,s_1)$ and $(p_2,y,s_2)$ with $x = x_1x_2x_3x_4, x_i \in \A$ and $y = y_1y_2y_3y_4$.  Let $\varphi(x_1) = u_{i_1}$, $\varphi(x_2) = u_{i_3}$, $\varphi(x_3) = u_{i_6}$, $\varphi(x_4) = u_{i_7}$ and $\varphi(y_1) = u_{i_2}$, $\varphi(y_2) = u_{i_4}$, $\varphi(y_3) = u_{i_5}$, $\varphi(y_4) = u_{i_8}$ (see the top line in Figure~\ref{fig:idea}). This structure can be decomposed into \emph{overhangs} (see Definition~\ref{dfn:overgang} and the second line in Figure~\ref{fig:idea}). Each overhang has left and right overlapping words, in the figure denoted by $s_1, s_2, s_3$ and $s_4$. The structure of these overhangs can be captured as a graph (see Definition~\ref{def:graph_of_overhangs} and the bottom line in Figure~\ref{fig:idea}): for instance the first overhang is connected with an directed edge with the second overhang since their right and left overlapping words are equal to $s_2$.
\begin{figure}[ht]
    \begin{center}\includegraphics{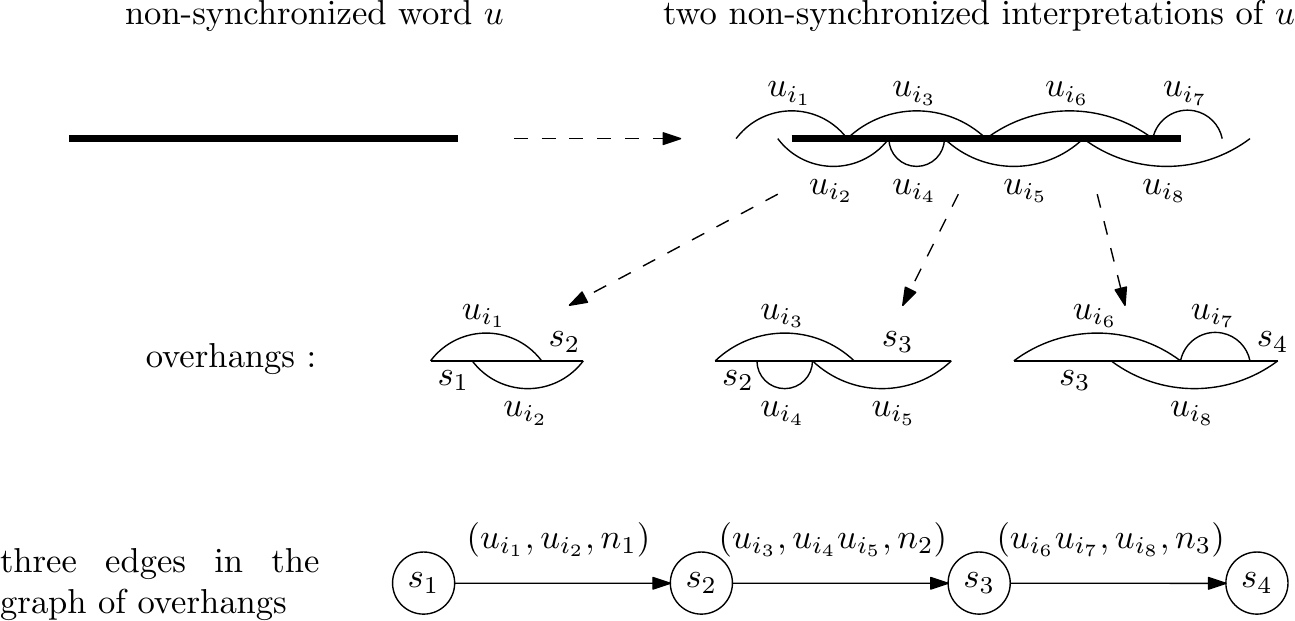}\end{center}
    \caption{The construction of a graph of overhangs, the numbers $n_1, n_2$ and $n_3$ are the lengths of overlaps of $u_{i_1}$ and $u_{i_2}$, $u_{i_3}$ and $u_{i_4}u_{i_5}$, and $u_{i_6}u_{i_7}$ and $u_{i_8}$, respectively.}\label{fig:idea}
\end{figure}
\begin{definition} \label{dfn:overgang}
	Given a PD0L-system $(\A, \varphi, w)$, let $X = \{\varphi(a) \mid a \in \A\}$. An ordered triplet
	$(u_1u_2\cdots u_m,v_1v_2 \cdots v_n, |x|)$, where $x$ is a non-empty word over $\A$ and $u_i$ and $v_j$ are elements of $X$ for all $1 \leq i \leq m$ and $1 \leq j \leq n$, is called an \emph{overhang} if the following conditions are satisfied:
	\begin{itemize}
	  \item[(i)] $x$ is a suffix of $u_1u_2\cdots u_m$ but not of $u_2u_3\cdots u_m$, 
	  \item[(ii)] $x$ is a prefix of $v_1v_2\cdots v_n$ but not of $v_1v_2\cdots v_{n-1}$,
	  \item[(iii)] either $x \neq  u_1u_2\cdots u_m$ or $x \neq v_1v_2\cdots v_n$,
	  \item[(iv)] $|v_1\cdots v_{n-1}| < |x(u_2\cdots u_m)^{-1}|$\,.
	\end{itemize}
	Here again we put $v_1\cdots v_{n-1} = \varepsilon$ if $n = 1$ and $u_2\cdots u_m = \varepsilon$ if $m = 1$.
	
	The word $x$ is called the \emph{common factor} of the overhang and the words
	$u_1u_2\cdots u_mx^{-1}$ and $x^{-1}v_1v_2\cdots v_{n}$ are \emph{left} and
	\emph{right overhang}, respectively.
\end{definition}
\begin{example}\label{PKonstrukceGrafu}
Consider the PD0L-system $G = (\{0, 1, 2\}, \varphi, 0)$ with $\varphi$ given by $0 \to 011, 1 \to 1120, 2 \to 120$.
The set $X$ in the definition above reads $\{011, 1120, 120\}$. To find all possible overhangs, we have to go through all elements $(u, v)$ of $X \times X$ and look for a non-empty word $x$ which is both a suffix of $u$ and prefix of $v$ and for $z \in X^*$ such that $(uz,v, |x|)$ or $(u,zv, |x|)$ is an overhang. After doing so we find all the overhangs for D0L-system $G$: $(011,1120, 2)$,  $(011, 1120, 1)$,  $(011 \, 120, 1120, 4)$, $(011, 120, 1)$, $(1120, 011, 1)$, $(120, 011, 1)$.
\end{example}
\begin{definition} \label{def:graph_of_overhangs}
	Given a PD0L-system $G = (\A, \varphi, w)$. A \emph{graph of overhangs} $GO_G$ for the PD0L-system $G$
	is given by this rule: words $s_1$ and $s_2$ are vertices of $GO_G$ connected with a directed edge $(s_1, s_2)$ labelled by $(u_1u_2\cdots u_m,v_1v_2 \cdots v_n, |x|)$ if the triplet $(u_1u_2\cdots u_m,v_1v_2 \cdots v_n, |x|)$ is an overhang such that $s_1$ is its
	left and $s_2$ its right overhang.
\end{definition}
\begin{example}[continued]
	The graph of overhangs for the PD0L-system from Example~\ref{PKonstrukceGrafu} is depicted in Figure~\ref{fig:graph_example}.
	\begin{figure}[!ht]
	\begin{center}
		\includegraphics{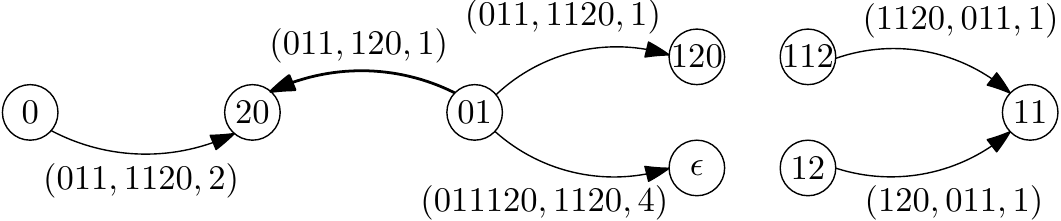}
		\caption{The graph of overhangs for Example~\ref{PKonstrukceGrafu}.} \label{fig:graph_example}
	\end{center}
	\end{figure}
\end{example}
To demonstrate usefulness of graphs of overhangs for the study of synchronizing delay, we first consider the simpler case of circular codes.
\begin{lemma}\label{lem:circular-code}
	Given a PD0L-system $G = (\A, \varphi, w)$. The set $X = \{\varphi(a) \mid a \in \A\}$ is a code if and only if the graph of overhangs $GO_G$ does not contain a cycle containing the empty word as a vertex.

	The set $X$ is a circular code if and only if its graph of overhangs does not contain any cycle.
\end{lemma}

\begin{proof}
	We prove the second part of the statement, the first part can be proved analogously.
	
	Assume that the graph of overhangs contains a cycle
$$
s_1 \xrightarrow[(v_1, u_1, \vert x_1 \vert)]{} s_2  \xrightarrow[(v_2, u_2, \vert x_2 \vert)]{} s_3 \cdots  s_\ell  \xrightarrow[(v_\ell, u_\ell, \vert x_\ell \vert)]{} s_1\, .
$$
For the word $y = s_1x_1s_2 \ldots s_\ell x_\ell$ we have 
$$
y = s_1x_1s_2 \ldots s_\ell x_\ell = v_1v_2 \cdots v_\ell = s_1u_1u_2 \ldots u_\ell (s_1)^{-1} \, ,
$$
where $v_i, u_j \in X^+$. It follows that $X$ is not a circular code by definition.

If $X$ is not a circular code, there is a word $z \in X^+$ such that
$$
z = v_1v_2 \ldots v_n = su_1u_2 \ldots u_m(s)^{-1} \, ,
$$ 
where $v_i, u_j \in X$ and $s$ is a suffix of $u_m$. These two decompositions of $z$ give an analogue of two non-synchronized interpretation of $z$. Using the same procedure as in Figure~\ref{fig:idea}, we can find numbers $1 \leq n_1 < \cdots < n_{k-1} < n$ and  $1 \leq m_1 < \cdots < m_{k-1} < m$ and words $x_i$ and $s_i$, $1 \leq i \leq k - 1$, so that we get the following cycle in the graph of overhangs:
$$
s \xrightarrow[(v_1 \cdots v_{n_1}, u_1 \cdots u_{m_1}, \vert x_1 \vert)]{} s_2 \cdots s_{k-1} \xrightarrow[(v_{n_{k-1}} \cdots v_n,u_{m_{k-1}} \cdots u_m , \vert x_{k-1} \vert)]{} s \, .
$$ 
\end{proof}
To explain the connection between circular codes and circular D0L-systems, consider a D0L-system $G = (\A, \varphi, w)$ that is not circular. It implies that for any $L \in \N$ there must be a word longer than $L$ with two non-synchronized interpretations. These two non-synchronized interpretations can be decomposed into overhangs as in Figure~\ref{fig:idea}. If $L$ is big enough, there must be two overhangs so that the left overhang of one of them  equals the right overhang of the other. Therefore they can be glued together to form a cycle which means that the set $X = \{\varphi(a) \mid a \in \A\}$ is not a circular code.

By Lemma~\ref{lem:circular-code} and by the previous paragraph, if a D0L-system $G$ is not circular, then $GO_G$ must contain an infinite walk and this happens if and only if $GO_G$ contains a cycle. However, existence of a cycle in $GO_G$ is not a sufficient condition as follows from the following example. 
\begin{example}
It is known that Thue-Morse D0L-system $G_{TM} = (\{0, 1\}, \varphi_{TM}, 0)$, where $\varphi_{TM}(0) = 01$ and $\varphi_{TM}(1) = 10$, is circular with minimal synchronizing delay $Z_\text{min} = 3$. However, its graph of overhangs contains two cycles, see Figure~\ref{OPrikladGrafTM}. These to cycles corresponds to words $(01)^j$ or $(10)^j$ for $j = 1,2,3,\ldots$, but such words are in $S(L(G_{TM}))$ only for $j \leq 2$.
\begin{figure}[!h]
\begin{center}
\includegraphics{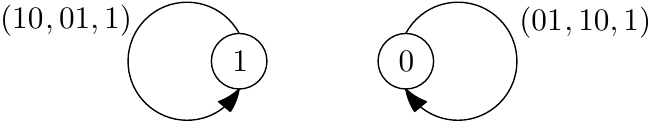}
\caption{Graph of overhangs of the D0L-system $G_{TM} = (\{0,1\}, \varphi_{TM}, 0)$.} \label{OPrikladGrafTM}
\end{center}
\end{figure}
\end{example}

\begin{definition}\label{def:g-admissible}
	Let $\ell > 1$ and
	$$
s_1 \xrightarrow[(v_1, u_1, \vert x_1 \vert)]{} s_2  \xrightarrow[(v_2, u_2, \vert x_2 \vert)]{} s_3 \cdots  s_{\ell} \xrightarrow[(v_\ell, u_\ell, \vert x_\ell \vert)]{} s_{\ell + 1}
$$
be a walk in the graph of overhangs of a PD0L-system $G = (\A, \varphi, w)$. We call this walk \emph{\mbox{$G$-admissible}} if both the words $v_1 \cdots v_\ell$ and $u_1 \cdots u_\ell$ are in $S(L(G))$ and $u_i, v_i \in \{\varphi(a) \mid a \in \A\}^+$ for all $1 \leq i \leq \ell$.

The number $|(s_1)^{-1}v_1\cdots v_\ell|$ is called the \emph{word-length} of the walk.
\end{definition}

With these definitions, we can state a sufficient and necessary condition for $G$ being non-circular: $G$ is non-circular if for any $L \in \N$ there is a $G$-admissible walk in $GO_G$ of word-length greater than $L$. Moreover, if $G$ is circular, then the word-longest $G$-admissible walk is tightly connected with the longest non-synchronized factor and so with the constant $Z_\text{min}$.
\begin{lemma}\label{lem:longest-walk}
	Let $G$ be a circular PD0L-system. Let $L_\text{max}$ denote the maximum $L$ such that there is a $G$-admissible walk in $GO_G$ of word-length $L$. It holds that
	$$
		L_\text{max} \leq Z_\text{min} \leq L_\text{max} + 2M - 3,
	$$
	where $M = \max_{a \in \A} |\varphi(a)|$.
\end{lemma}
\begin{proof}
	Let the word-longest walk in $GO_G$ be denoted by
	$$
p = s_1 \xrightarrow[(v_1, u_1, \vert x_1 \vert)]{} s_2  \xrightarrow[(v_2, u_2, \vert x_2 \vert)]{} s_3 \cdots  s_{\ell} \xrightarrow[(v_\ell, u_\ell, \vert x_\ell \vert)]{} s_{\ell + 1}.
$$
By the construction of $GO_G$ we have that $(s_1)^{-1}v_1\cdots v_\ell$ is non-synchronized. Since $Z_\text{min}$ equals the length of the longest non-synchronized factor from $S(L(G))$, it holds that $|(s_1)^{-1}v_1\cdots v_\ell| = L_\text{max} \leq Z_\text{min}$. 
To obtain an upper bound, consider the longest possible word $x$ such that $xu_1\cdots u_\ell(s_{\ell+1})^{-1}$ is a non-synchronized factor in $S(L(G))$. It is easy to see that if $|x| \geq M-1$, then we can find a word-longer walk in $GO_G$ than $p$. Therefore $|x| \leq M - 2$. Analogously we prove that if $(s_1)^{-1}v_1\cdots v_\ell y$ is non-synchronized, than $|y| \leq M - 1$ (see Figure~\ref{fig:longest-walk}).
\begin{figure}[!h]
\begin{center}
\includegraphics{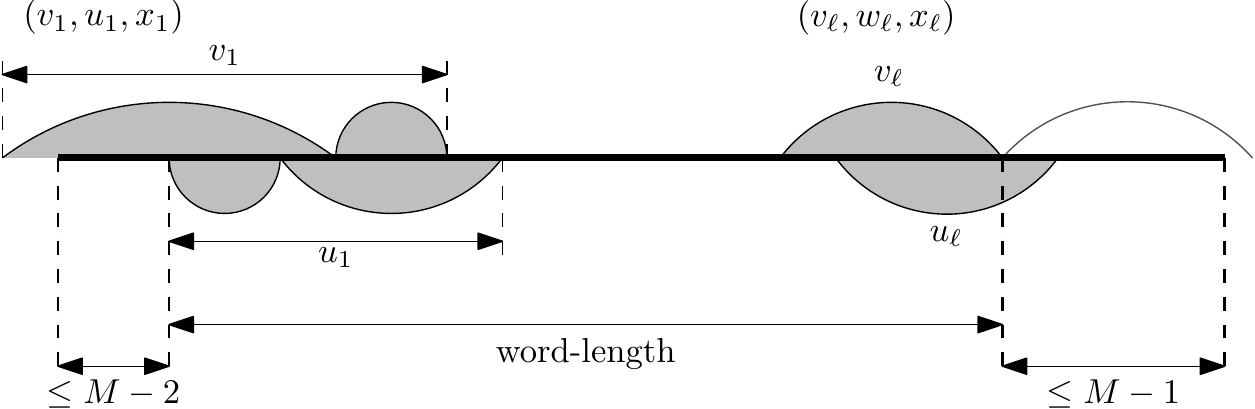}
\caption{The bold line denotes the longest possible prolongation of a non-synchronized factor corresponding to the word-longest walk in $GO_G$.} \label{fig:longest-walk}
\end{center}
\end{figure}
\end{proof}
In the following sections we use the last lemma to get an upper bound on $Z_\text{min}$ by finding the word-longest possible walk in $GO_G$ for $G$ with uniform binary morphisms. 

\section{Proof of the main result} \label{sec:the-proof}

There is no known estimate on the (minimal) synchronizing delay of a PD0L-system. Here we restrict ourselves to the simplest class of morphisms: uniform morphisms defined on a binary alphabet. We further assume that the axiom of the D0L-system is a letter; considering a general case would not bring much novelty but technical difficulties.
\begin{definition}
	A morphism $\varphi$ on the binary alphabet $\{a,b\}$ is \emph{$k$-uniform} for some integer $k \geq 2$ if $|\varphi(a)| = |\varphi(b)| = k$.
\end{definition}
Because of the simple structure of D0L-systems with a $k$-uniform morphism, we can find an explicit list of morphisms that are not circular. We state a proof here as it is quite simple using the results from~\cite{KlSt13}. However, a complete characterisation of repetitive (and so non-circular) D0L-systems with binary morphisms was done in~\cite{KoOtSe97}. Partial characterisation of circular D0L-systems with a binary $k$-uniform morphism was also given in~\cite{Fr98, Fr98_DOL}. 
\begin{lemma} \label{LNoncircularMorfism}
	A D0L-system $G = (\{a, b\}, \varphi, a)$ with $k$-uniform morphism $\varphi$ is not circular if and only if one of the following conditions is satisfied: 
\begin{enumerate}[(i)]
\item $\varphi(a) = \varphi(b)$,
\item $\varphi(a) = a^k$ or $\varphi(b) = b^k$, 
\item $\varphi(a) = b^k$ and $\varphi(b) = a^k$,
\item $k = 2m + 1, m \geq 1$, and $\varphi(a) = (ab)^ma$, $\varphi(b) = (ba)^mb$,
\item $k = 2m + 1, m \geq 1$, and $\varphi(a) = (ba)^mb$, $\varphi(b) = (ab)^ma$.
\end{enumerate}
\end{lemma}

\begin{proof}
	As explained above, circular D0L-systems are not unboundedly repetitive. Since the D0L-system $G$ cannot contain a bounded letter (as it is $k$-uniform with $k > 1$), it is unboundedly repetitive if and only if it is repetitive. Hence, we need to find all $\varphi$ for which the D0L-system is non-injective or repetitive. Clearly, a $k$-uniform morphism $\varphi$ is not injective if and only if $\varphi(a) = \varphi(b)$ (condition $(i)$). It follows from~\cite{KlSt13} that $G$ is repetitive if and only if there is a primitive word $u$ (i.e., a word for which $u = z^\ell$ implies $\ell = 1$) in $S(L(G))$ such that $\varphi^j(u) = u^\ell$ for some integers $j \geq 1$ and $\ell \geq 2$. Moreover, the factor $u$ cannot contain any unbounded letter twice and $j$ is bounded by the number of letters in the alphabet: hence we must have $u \in \{a,b,ba,ab\}$ and $j \in \{1,2\}$. Analysing all these eight possible cases we obtain conditions $(ii)-(v)$.
\end{proof}

\subsection{List of all possible components of graphs of overhangs}

In general, even for a simple-looking morphism over a small alphabet the corresponding graph of overhangs can be of very complex structure. However, in the case of (binary) uniform morphisms, these graphs are always divided into components containing relatively small number of vertices; for a binary uniform morphism this number is at most four as follows from these two simple observations:
\begin{lemma}
Given a D0L-system $G = (\A, \varphi, w)$ with a $k$-uniform morphism $\varphi$. For the graph of overhangs of $G$ the following holds:
\begin{enumerate}[(i)]
	\item A label of any edge is of the form of $(\varphi(a), \varphi(b), \ell)$ with $a,b \in \A$ and $0 < \ell < k$.
	\item If there is an edge from a vertex $s_1$ to a vertex $s_2$ with label  $(\varphi(a), \varphi(b), \ell)$, then the words $s_1$ and $s_2$ are of the same length $k - \ell$.
\end{enumerate}
\end{lemma}
Applied on the case of binary morphism we obtain this:
\begin{corollary}\label{col:components}
	Given a D0L-system $G = (\A, \varphi, w)$ with an injective $k$-uniform morphism $\varphi$ defined over the binary alphabet $\A = \{a, b\}$. The graph of overhangs $GO_G$ consists of weakly connected components each containing at most four vertices. Labels of these vertices are all of the same length $k - \ell$, $0 < \ell < k$, and the edges are labelled by four possible labels: $(\varphi(a), \varphi(a), \ell), (\varphi(a), \varphi(b), \ell), (\varphi(b), \varphi(a), \ell), (\varphi(b), \varphi(b), \ell)$.
\end{corollary}
\begin{proof}
	It is a direct consequence of the previous lemma that $GO_G$ is divided into weakly connected components and also that within each component all possible labels of its edges are four overhangs $(\varphi(a), \varphi(a), \ell), (\varphi(a), \varphi(b), \ell), (\varphi(b), \varphi(a), \ell), (\varphi(b), \varphi(b), \ell)$. Further, the left overhangs of $(\varphi(a), \varphi(a), \ell)$ and $(\varphi(a), \varphi(b), \ell)$ are the same words. The same is true for $(\varphi(b), \varphi(a), \ell)$ and $(\varphi(b), \varphi(b), \ell)$. Similarly, the right overhangs are equal words for $(\varphi(a), \varphi(a), \ell)$ and $(\varphi(b), \varphi(a), \ell)$ and for $(\varphi(a), \varphi(b), \ell)$ and  $(\varphi(b), \varphi(b), \ell)$ (see Figure~\ref{fig:maximal_component}). It follows that each component has at most four vertices.
	\begin{figure}[ht]
	\begin{center}
		\includegraphics{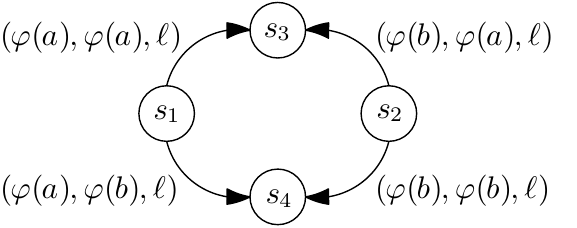}
		\caption{An illustration of a weakly connected component of a graph of overhangs for a binary $k$-uniform morphism.} \label{fig:maximal_component}
	\end{center}
	\end{figure}
\end{proof}

As explained above, we are interested in the word-longest $G$-admissible walks (see Definition~\ref{def:g-admissible}) in the graphs of overhangs. The previous corollary says that it is sufficient to look for such walks only in components (in what follows by components we always mean the weakly connected components from Corollary~\ref{col:components}) with at most four vertices. The component in Figure~\ref{fig:maximal_component} can be considered as maximal, since all other components that can appear in the graph of overhangs for an injective binary $k$-uniform morphism arise by merging some vertices of it and by omitting some edges. In other words, all possible components are subgraphs of the graphs depicted in Figures~\ref{fig:maximal_component}, \ref{fig:components_with_3_vertices}, \ref{fig:components_with_2_vertices} and~\ref{fig:components_with_1_vertex} (labels are denoted by $\frac{\varphi(x)}{\varphi(y)}$ instead of $(\varphi(x), \varphi(y), \ell)$ in order to save some space).
\begin{figure}[ht]
	\begin{center}
		\includegraphics{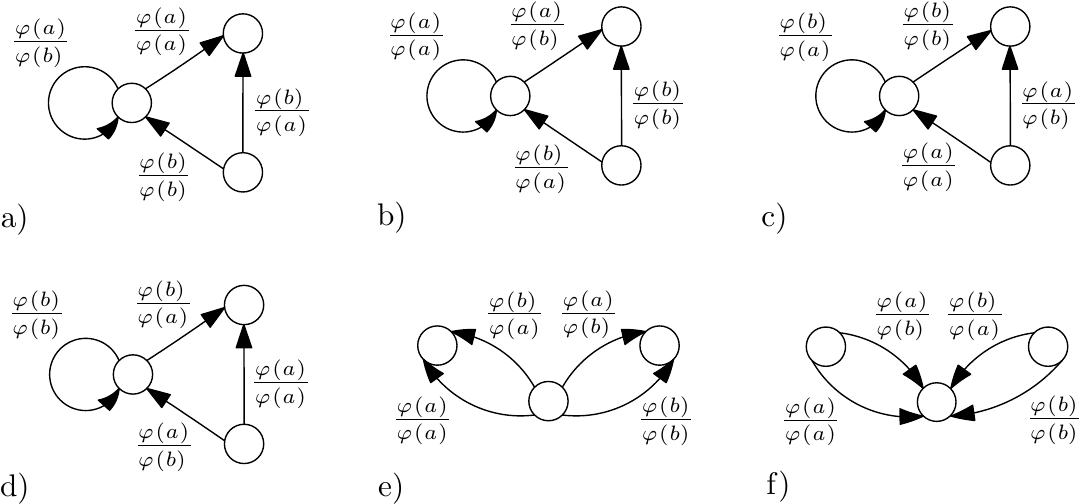}
		\caption{All possible components with three vertices arise by ommiting edges in one of these six graphs.} \label{fig:components_with_3_vertices}
	\end{center}
\end{figure}
\begin{figure}[ht]
	\begin{center}
		\includegraphics{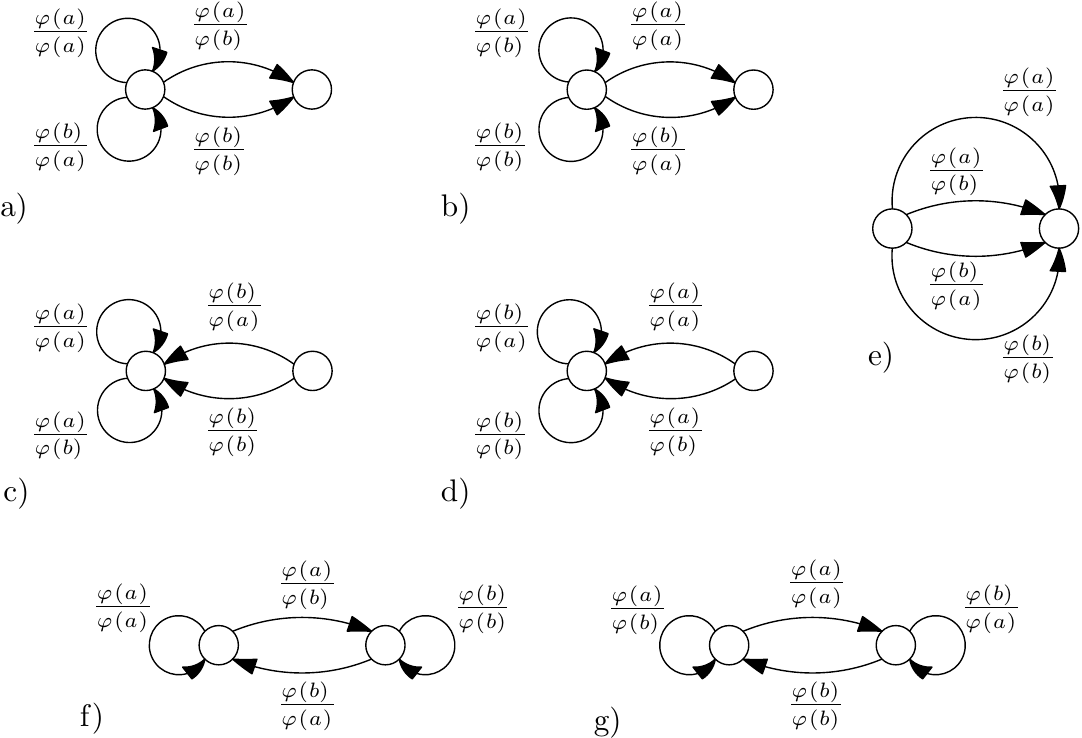}
		\caption{All possible components with two vertices arise by ommiting edges in one of these seven graphs.} \label{fig:components_with_2_vertices}
	\end{center}
\end{figure}
\begin{figure}[ht]
	\begin{center}
		\includegraphics{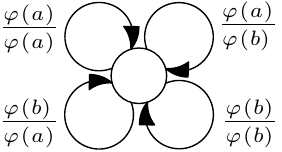}
		\caption{All possible components with one vertex arise by ommiting edges in this graph.} \label{fig:components_with_1_vertex}
	\end{center}
\end{figure}

\subsection{Infeasible subgraphs}

Here we prove that the components of the graph of overhangs of circular D0L-systems with a binary $k$-uniform morphism cannot contain the following subgraphs: two loops on one vertex (Figure~\ref{fig:prohibited_structures} a)), a vertex with a loop that is at the same time a vertex of a cycle over two vertices (Figure~\ref{fig:prohibited_structures} b)) and two vertices with a loop that are connected with an edge (Figure~\ref{fig:prohibited_structures} c)). All the proofs are quite similar: we use the following three auxiliary lemmas to get a contradiction with circularity; namely, we show that these subgraphs appear in the graphs only for non-cicular D0L-systems, i.e., only for morphisms listed in Lemma~\ref{LNoncircularMorfism}.
\begin{figure}[!ht]
\begin{center}
\includegraphics{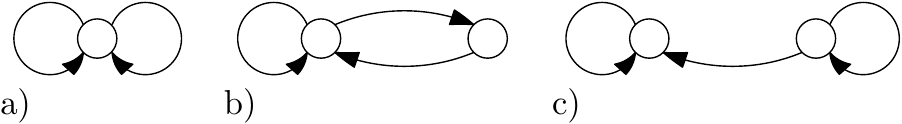}
\caption{The structures that never appear in the graphs of overhangs of a circular D0L-systems with a binary $k$-uniform morphism.} \label{fig:prohibited_structures}
\end{center}
\end{figure}
%

\begin{lemma}[Lyndon, Sch\"{u}tzenberger~\cite{LynSch62}] \label{LLyndon}
Let $x, y \in \A^+$. The following three conditions are equivalent:
\begin{enumerate}[(i)]
\item $xy  = yx$;
\item There exist integers $i, j > 0$ such that $x^i = y^j$;
\item There exist $z \in \A^+$ and integers $p, q > 0$ such that $x = z^p$ and $y = z^q$.
\end{enumerate}
\end{lemma}

\begin{lemma}[Krieger~\cite{Kr09}]\label{L12}
Let $x, y, z ,t \in \A^+$. Assume the following equalities hold:
\begin{enumerate}[(i)]
\item $xy = zt$ (equivalently, $ty = zx$);
\item $yx = xz$ .
\end{enumerate}
Then $y = z$ and $x = t$. 
\end{lemma} 

\begin{lemma}[Krieger~\cite{Kr09}] \label{L13}
Let $x, y, z ,t \in \A^+$. Assume the following equalities hold:
\begin{enumerate}[(i)]
\item $xy = yt$;
\item $tz = zx$.
\end{enumerate}
Then there exists $u \in \A^+, v \in \A^*$ and integers $i \geq 1$ and $j, m \geq 0$ such that $x = (uv)^i$, $t = (vu)^i$, $y = (uv)^ju$, $z = (vu)^mv$. If in addition $\vert y \vert = \vert z \vert$, then either $v = \epsilon$ and $m = j+1$, or $\vert u \vert = \vert v \vert$ and $m = j$.
\end{lemma} 

\begin{lemma}
Let $G = (\A, \varphi, a)$ be a circular D0L-system with a binary $k$-uniform morphism $\varphi$ and $GO_G$ its graph of overhangs. Then $GO_G$ does not contain two loops on a common vertex (see Figure~\ref{fig:prohibited_structures} a)) as its subgraph. 
\end{lemma}
\begin{proof}
Assume that there is a component containing a vertex $s$ such that $|s| = k - \ell$ and there are two loops starting and ending in $s$. We have six possible pairs of labels for the two edges of the loops (see Figure~\ref{fig:components_with_1_vertex}). We consider only the following three cases since the others can be handled analogously.

First assume that the two loops on $s$ are labelled with $(\varphi(a), \varphi(a), \ell)$ and $(\varphi(b), \varphi(a), \ell)$. We must have $\varphi(a) = sy = ys = zs$ and $\varphi(b) = sz$. It follows that $z = y$ and therefore $\varphi(a) = \varphi(b)$ and $G$ is not circular.

Now assume the two labels read $(\varphi(a), \varphi(a), \ell)$ and $(\varphi(b), \varphi(b), \ell)$. We get $\varphi(a) = sy = ys$ and $\varphi(b) = sz = zs$. By Lemma~\ref{L13} applied to $x = t = s$ we must have $s = (uv)^i = (vu)^i$, $y = (uv)^ju$ and $z = (vu)^mv$ with integers $i \geq 1, j, m \geq 0$. Since $uv = vu$, Lemma~\ref{LLyndon} (iii) says that both $u$ and $v$ are powers of the same word and hence we have again  $\varphi(a) = \varphi(b)$.

Finally, assume the labels are $(\varphi(a), \varphi(b), \ell)$ and $(\varphi(b), \varphi(a), \ell)$. It holds that $\varphi(a) = sy = zs$ and $\varphi(b) = sz = ys$. By Lemma~\ref{L12} applied for $x = t = s$ we have $y = z$ and so again $\varphi(a) = \varphi(b)$.
\end{proof}

\begin{lemma} \label{Lprohibitedcycle}
Let $G = (\A, \varphi, a)$ be a circular D0L-system with a binary $k$-uniform morphism $\varphi$ and $GO_G$ its graph of overhangs. Then $GO_G$ does not contain a cycle on two vertices with a loop on one of them (see Figure~\ref{fig:prohibited_structures} b)) as its subgraph. 
\end{lemma}

\begin{proof}
The proof is quite similar to the previous one: we again assume that the structure can be found in the graph on vertices $s_1$ and $s_2$. The words $s_1$ and $s_2$ must be of the same length which is denoted by $k - \ell$. As follows from the Figure~\ref{fig:components_with_2_vertices} f) and g), the labels of the cycle are either $(\varphi(a), \varphi(a), \ell)$ and  $(\varphi(b), \varphi(b), \ell)$ or $(\varphi(a), \varphi(b), \ell)$ and  $(\varphi(b), \varphi(a), \ell)$.

Assume the former case is true, the latter one can be handled analogously. We get $\varphi(a) = s_1y = ys_2$, $\varphi(b) = s_2z = zs_1$ and $|y| = |z|$. Due to Lemma~\ref{L13} we can write $\varphi(a) = (uv)^{i+j}u$, $\varphi(b) = (vu)^{i+m}v$, $s_1 = (uv)^i$ and $s_2 = (vu)^i$ for some $u \in \A^+, v \in \A^*$ and integers $i \geq 1$ and $j, m \geq 0$. If $v = \epsilon$ and $m = j+1$, then $\varphi(a) = \varphi(b)$ which contradicts circularity, therefore we must have $\varphi(a) = (uv)^{i+j}u$, $\varphi(b) = (vu)^{i+j}v$ with $u$ and  $v$ non-empty and of the same length. 

Assume that the loop is on the vertex $s_1$ (the other case is again analogous), the label of the loop must read $(\varphi(a), \varphi(b), \ell)$. Such an overhang can exist only if $uv = vu$. This is again a contradiction since $\varphi(a) = \varphi(b)$.    
\end{proof}

\begin{lemma}
Let $G = (\A, \varphi, a)$ be a circular D0L-system with a binary $k$-uniform morphism $\varphi$ and $GO_G$ its graph of overhangs. Then $GO_G$ does not contain two vertices with a loop that are also connected with an edge (see Figure~\ref{fig:prohibited_structures} c)) as its subgraph. 
\end{lemma}

\begin{proof}
	This structure can appear in $GO_G$ only as a subgraph of the graphs f) and g) in Figure~\ref{fig:components_with_2_vertices}.
	
	Assume the loops are on the vertices $s_1$ with label $(\varphi(a), \varphi(a), \ell)$ and on $s_2$ with label $(\varphi(b), \varphi(b), \ell)$ and the edge from $s_1$ to $s_2$ is labelled by $(\varphi(a), \varphi(b), \ell)$ (the other case can be obtained just by exchanging the letters $a$ and $b$). It must hold that $\varphi(a) = s_1x = xs_1 = s_1y$ and $\varphi(b) = ys_2 = s_2z = zs_2$. It follows that $x = y = z$ and so $\varphi(a) = \varphi(b)$ by Lemma~\ref{LLyndon} $(iii)$. Hence $G$ is not injective and so non-circular.
	
	Assume now the loops are on the vertex $s_1$ with label $(\varphi(a), \varphi(b), \ell)$ and on the vertex $s_2$ with label $(\varphi(b), \varphi(a), \ell)$ and the edge from $s_1$ to $s_2$ is labelled with $(\varphi(a), \varphi(a), \ell)$ (the other case can be obtained just by exchanging the letters $a$ and $b$). It holds that $\varphi(a) = s_1x = s_1y = ys_2 = zs_2$ and $\varphi(b) = xs_1 = s_2z$. It follows that $x = y = z$ and by Lemma~\ref{L12} applied for x = t we have $\varphi(a) = \varphi(b)$.
\end{proof}

\subsection{Word-longest walks}

If we put all the previous results together, we can say that any walk in the graph of overhangs of a circular D0L-system $G$ with a binary $k$-uniform morphism is of one of the shapes depicted in Figure~\ref{fig:possible_walks}.
\begin{figure}[!ht]
\begin{center}
\includegraphics{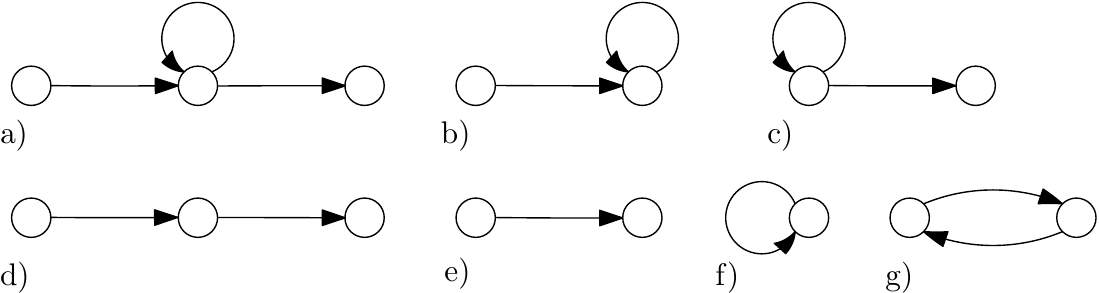}
\caption{All possible walks in graphs of overhangs of circular D0L-systems with a binary $k$-uniform morphism.} \label{fig:possible_walks}
\end{center}
\end{figure}
As we are looking for the word-longest $G$-admissible walk, it suffices to focus on walks in graphs a) and g) in Figure~\ref{fig:possible_walks}. However, the word-length of a $G$-admissible walk depends on how many times such a walk goes along the edge of the loop (in case a)) or along the both edges of the cycle (case g)). Denote the maximal number of those transitions along the edges by $R_1$ and $R_2$, respectively. We recall that these numbers must be finite for circular D0L-systems.

Denote the three vertices in graph a) in Figure~\ref{fig:possible_walks} by $s_1, s_2, s_3$ ($s_2$ is the vertex with the loop) and the common factors of the labels of the respective edges by $x_1, x_2, x_3$. We set 
$R_1$ to the maximum $R \in \N$ such that the walk
$$
\underbrace{s_2 \rightarrow s_2 \rightarrow \cdots \rightarrow s_2}_{R \text{ arrows}}
$$
is $G$-admissible for some circular D0L-system $G$ with a binary $k$-uniform morphism. The word-length of any $G$-admissible walk in any subgraph of graph a) in Figure~\ref{fig:possible_walks} is bounded by $|x_1(s_2x_2)^{R_1}s_2x_3| = |x_1|+ R_1k + k \leq R_1k + 2k - 1$.

Similarly, we denote the vertices of the cycle in graph g) in Figure~\ref{fig:possible_walks} by $s_1$ and $s_2$ and the common factors of the labels by $x_1$ and $x_2$. We set 
$R_2$ to the maximum $R \in \N$ such that the walk
$$
\underbrace{s_1 \rightarrow s_2 \rightarrow \cdots \rightarrow s_2 \rightarrow s_1}_{2R \text{ arrows}}
$$
is $G$-admissible for some circular D0L-system $G$ with a binary $k$-uniform morphism.  The word-length of any $G$-admissible walk in any subgraph of graph g) in Figure~\ref{fig:possible_walks} is bounded by $|(s_1)^{-1}(s_1x_1s_2x_2)^{R_2}s_1x_1| = 2kR_2 + |x_1| \leq 2kR_2 + k - 1$ (we have considered the walk that goes through the cycle $(R_2 + 1/2)$ times).

By Lemma~\ref{lem:longest-walk}, the minimum synchronizing delay $Z_\text{min}$ is equal to or less than the number $\max\{R_1k + 4k - 4, 2kR_2 + 3k - 4\}$. It remains to find an upper bound on $R_1$ and $R_2$. We start with $R_1$: the label of the loop in graph a) in Figure~\ref{fig:possible_walks} is one of these four overhangs: $(\varphi(a), \varphi(a), \ell)$, $(\varphi(a), \varphi(b), \ell)$,  $(\varphi(b), \varphi(b), \ell)$ and $(\varphi(b), \varphi(a), \ell)$. As the last two cases are analogous to the first or the second one, we focus only on labels $(\varphi(a), \varphi(a), \ell)$ and $(\varphi(a), \varphi(b), \ell)$: in the former case $R_1 = R_a$ and in the latter case $R_1 = \min\{R_a, R_b\}$, where $R_a = \max\{\ell \in \N : a^\ell \in S(L(G))\}$ and $R_b = \max\{\ell \in \N : b^\ell \in S(L(G))\}$.
\begin{lemma}\label{lem:loop1}
Let $G = \{\A, \varphi, a\}$ be a circular D0L-system with a binary $k$-uniform morphism $\varphi$ and $GO_G$ its graph of overhangs containing a vertex $s$ with a loop labelled with $(\varphi(a), \varphi(a), \ell)$. Then the following holds:
\begin{enumerate}
\item $R_a \leq k-1$ if $k$ is a prime number,
\item $R_a \leq k(\frac{k}{d}-1) + 1$ otherwise,
\end{enumerate}
where number $d$ is the least divisor of $k$ greater than $1$. These bounds are attained for some D0L-system.
\end{lemma}

\begin{proof}
We find a morphism $\varphi$ so that the value of $R_a$ for the resulting D0L-system is maximum. We have that $\varphi(a) = xy = yx$ and $\varphi(b)$ is arbitrary. By Lemma \ref{LLyndon}, there exists a word $z$ such that $\varphi(a) = z^{m}$ for some $z \in \A^+$ and integer $m \geq 2$. Since $G$ is circular, one of the following is true by Lemma~\ref{LNoncircularMorfism}:
\begin{enumerate}[(i)]
\item $\varphi(a) = z^m$, $\varphi(b) = a^k$, and $|z|_a > 0$,  $|z|_b > 0$. This is possible only if $k$ is not prime.
\item $\varphi(a) = z^m$, $\varphi(b) = w$, where $|z|_a > 0$,  $|z|_b > 0$, $|w|_a >0$ and $|w|_b >0$. This is possible only if $k$ is not prime, too.
\item $\varphi(a) = b^k$, $\varphi(b) = w$, where $|w|_a >0$ and $|w|_b >0$.
\end{enumerate}

We analyse these three cases separately.

Case $(i)$: Let $p, q$ denote the maximum numbers such that $a^p$ is a prefix of $\varphi(a)$ and $a^q$ is suffix of $\varphi(a)$. We have that $R_a = R_bk + p + q$.  Since $p+q < |z| \leq \frac{k}{2}$, the maximum value of $R_a$ is attained if we maximize the value of $R_b$ (regardless of the value of $p+q$). The letter $b$ is present only in $\varphi(a)$, thus $R_b$ is given by the maximal power of $b$ in the word $zz$ (note that $z$ contains the letter $a$). So $R_b$ is maximum possible if $m = d$, where $d$ is the least divisor of $k$ greater than $1$. We get $\vert z \vert = \frac{k}{d}$ and so the maximum possible value of $R_b$ is $\frac{k}{d} - 1$. It follows that $p+q \leq 1$ and $R_a \leq k\left(\frac{k}{d}-1\right) + 1$.     
This bound is attained for $\varphi(a) = (ab^{\frac{k}{d}-1})^d$, $\varphi(b) = a^k$.

Case $(ii)$: We have $|z|_a \leq \frac{k}{d}-1$ and $|\varphi(b)|_a \leq k-1$, where $d$ is again the least divisor of $k$ greater than $1$. The value of $R_a$ is equal to the maximum value of $\ell$ such that $a^\ell$ is a factor of $\varphi(a)\varphi(a)$, $\varphi(a)\varphi(b)$, $\varphi(b)\varphi(a)$ or $\varphi(b)\varphi(b)$. It is easy to see that $R_a \leq k-1 + \frac{k}{d} - 1 = \frac{k(d+1)}{d} -2$.

Case $(iii)$: The letter $a$ is present only in the word $\varphi(b) = w$, thus the maximum power of $a$ in $S(L(G))$ is equal to the maximum power of $a$ that appears in the word $ww$. Since $|w|_b >0$, we get $R_a \leq k-1$. This bound is attained for $\varphi(a) = b^k$, $\varphi(b) = ba^{k-1}$.

If $k$ is a prime number, the proof is finished with $R_a \leq k - 1$. Let $k$ be not prime and $d$ its least divisor greater than one. It is an easy exercise to show that
$$
	k - 1 < k \leq \frac{k(d+1)}{d} - 2 < k \left(\frac{k}{d} - 1\right) + 1,
$$
which concludes the proof.
\end{proof}

\begin{lemma}\label{lem:loop2}
Let $G = \{\A, \varphi, a\}$ be a circular D0L-system with a binary $k$-uniform morphism $\varphi$ and $GO_G$ its graph of overhangs containing a vertex $s$ with a loop labelled with $(\varphi(a), \varphi(b), \ell)$. Then for $R_1 = \max\{ R_a, R_b \}$ the following holds:
\begin{enumerate}
\item $R_1 \leq k$ if $k$ is even,
\item $R_1 \leq k-1$ if $k$ is odd.
\end{enumerate}
These bounds are attained for some D0L-system.
\end{lemma}

\begin{proof}
It follows from the assumptions that $\varphi(a) = xy$, $\varphi(b) = yx$, where $x, y$ are non-empty words, and $|\varphi(a)|_a$ = $|\varphi(b)|_a > 0$ and $|\varphi(a)|_b$ = $|\varphi(b)|_b > 0$. 

To find the largest powers of $a$ and $b$ it suffices to go through this four factors of length $2k$: $\varphi(a)\varphi(a)$, $\varphi(a)\varphi(b)$, $\varphi(b)\varphi(a)$ and $\varphi(b)\varphi(b)$. We must have $R_a + R_b \leq 2k$. Since $R_1 = \min \{R_a, R_b\}$, its maximal possible value for even $k$ is attained if $R_a = R_b = k$, if $k$ is odd, the maximum value of $R_1$ is $k - 1$. These maximum values of $R_1$ are attained  for the morphism $\varphi(a) = a^{\lfloor\frac{k}{2}\rfloor}b^{\lceil\frac{k}{2}\rceil}$, $\varphi(b) = b^{\lceil\frac{k}{2}\rceil}a^{\lfloor\frac{k}{2}\rfloor}$.
\end{proof}

It remains to consider the cycle (i.e., case g) in Figure~\ref{fig:possible_walks}). The labels of the cycle are either $(\varphi(a), \varphi(a), \ell)$ and $(\varphi(b), \varphi(b), \ell)$ or $(\varphi(a), \varphi(b), \ell)$ and $(\varphi(b), \varphi(a), \ell)$. We consider only the former case since the latter one is analogous. Let $R_{ab} = \max\{\ell \in \N : (ab)^\ell \in S(L(G))\}$ and $R_{ba} = \max\{\ell \in \N : (ba)^\ell \in S(L(G))\}$, then $R_2 = \min\{R_{ab},R_{ba}\}$.

\begin{lemma}
Let $G = \{\A, \varphi, a\}$ be a circular D0L-system with binary $k$-uniform morphism $\varphi$ and $GO_G$ its graph of overhangs containing a cycle with labels $(\varphi(a), \varphi(a), \ell)$ and $(\varphi(b), \varphi(b), \ell)$. Then it holds that $R_2 \leq \frac{k-2}{2}$.
\end{lemma}

\begin{proof}

As explained in the proof of Lemma~\ref{Lprohibitedcycle}, we know that $\varphi(a) = (uv)^{i+j}u$, $\varphi(b) = (vu)^{i+j}v$ for some words  $u, v \in \A^+$, $|u| = |v|$, $u \neq v$ and integers $i \geq 1, j \geq 0 $. By Lemma \ref{LNoncircularMorfism}, one of the words $\varphi(a)$ or $\varphi(b)$ must contain $aa$ or $bb$ as a factor (otherwise the D0L-system $G$ is not circular). Thus, some of the factors $u$, $v$, $uv$, $vu$ contain $aa$ or $bb$. In order to find the maximum number of repetitions of $ab$ and $ba$ in $S(L(G)$, we can restrict ourselves to words $\varphi(a)\varphi(a)$, $\varphi(a)\varphi(b)$, $\varphi(b)\varphi(a)$ and $\varphi(b)\varphi(b)$. Obviously numbers $R_{ab}$ and $R_{ba}$ will be largest possible, if $u$ and $v$ are longest possible, i.e., $\varphi(a) = uvu$, $\varphi(b) = vuv$. Since $R_2 = \min \{R_{ab}, R_{ba}\}$, we get $R_2 \leq \frac{k-2}{2}$.   
\end{proof}

As explained at the beginning of this subsection the minimum synchronizing delay $Z_\text{min}$ is equal to or less than the number  $\max\{R_1k + 4k - 4, 2kR_2 + 3k - 4\}$. If $k = 2$, then the maximum value of $R_1$ is $2$ (by Lemma~\ref{lem:loop2}) and $R_2 = 0$ (as a cycle cannot appear in $GO_G$), it follows that $Z_\text{min} \leq 2k + 4k - 4 = 8$. 

If $k$ is an odd prime number, then the maximum value of $R_1$  is $k - 1$. Since $R_2$ is less than or equal to $(k - 2)/2$ we have
$$
	Z_\text{min} \leq  \max\{R_1k + 4k - 4, 2kR_2 + 3k - 4\} =  \max\{k^2 + 3k - 4, k^2 + k - 4\} = k^2 + 3k - 4.
$$

Finally, if $k$ is not prime and not equal to $2$, the maximal value of $R_1$ equals 
$$
	\max\left\{k\left(\frac{k}{d}-1\right) + 1, k \right\} = k\left(\frac{k}{d}-1\right) + 1,
$$
where $d$ is the smallest divisor of $k$ greater than one. Since $R_2$ is still less than or equal to $(k - 2)/2$, we have that 
$$
	Z_{\text{min}} \leq R_1k + 4k - 4 = k^2\left(\frac{k}{d}-1\right) + 5k - 4.
$$
This concludes the proof of Theorem~\ref{thm:main_result}.

\section{Conclusion}

The upper bounds in Theorem~\ref{thm:main_result} can be slightly improved: for instance if $k$ is a prime number, we have proved that any word containing more than $k^2(\frac{k}{d} -1) + \frac{k(d-1)}{d} + 1$ letters has a synchronizing point. This bound is attained for the morphism $\varphi(a) = (ab^{\frac{k}{d}-1})^d$, $\varphi(b) = a^k$. However, the proofs of these improved bounds are very technical and difficult to follow, therefore we do not state them here. The proof techniques presented in this paper could be used for uniform morphisms over three or more letter alphabets. Unfortunately, the higher the number of letters is, the higher the number of subgraphs of the respective graph of overhangs to be considered. In order to obtain some reasonable bound on minimal synchronizing delay for a general D0L-system, some other proof techniques must be used.


\section{Acknowledgment}

We acknowledge financial support by the Czech Science Foundation, grant GA\v CR 13-35273P (the second author), and by
the Grant Agency of the Czech Technical University in Prague, grant SGS11/162/OHK4/3T/14 (the first author).





\bibliographystyle{elsarticle-num}
\bibliography{biblio}

\end{document}